\documentclass[12pt]{amsart}

\usepackage{amssymb,amsthm,mathtools,bm,color,scalerel}
\usepackage{hyperref}

\def\P{{\mathbb{P}}}
\def\E{{\mathbb{E}}}
\def\Z{{\mathbb{Z}}}

\newtheorem{thm}{Theorem}

\usepackage{tikz}
\usetikzlibrary{positioning}
\usetikzlibrary{svg.path}
\definecolor{orcid_color}{HTML}{A6CE39}

\hypersetup{pdfborder={0 0 0}} 
\DeclareRobustCommand{\orcidicon}{%
	\raisebox{.2mm}{\scalerel*{%
	\begin{tikzpicture}[xscale=1,yscale=-1,transform shape]
	\filldraw[color=orcid_color] svg {M256,128c0,70.7-57.3,128-128,128C57.3,256,0,198.7,0,128C0,57.3,57.3,0,128,0C198.7,0,256,57.3,256,128z};
	\filldraw[color=white] svg {M86.3,186.2H70.9V79.1h15.4v48.4V186.2z} svg {M108.9,79.1h41.6c39.6,0,57,28.3,57,53.6c0,27.5-21.5,53.6-56.8,53.6h-41.8V79.1z M124.3,172.4h24.5
		c34.9,0,42.9-26.5,42.9-39.7c0-21.5-13.7-39.7-43.7-39.7h-23.7V172.4z} svg {M88.7,56.8c0,5.5-4.5,10.1-10.1,10.1c-5.6,0-10.1-4.6-10.1-10.1c0-5.6,4.5-10.1,10.1-10.1
		C84.2,46.7,88.7,51.3,88.7,56.8z};
	\end{tikzpicture}}{|}}%
}
\newcommand{\orcid}[1]{\href{https://orcid.org/#1}{\orcidicon}}

\title{Burning Hamming graphs}
\author[Norihide Tokushige]{Norihide Tokushige\,\orcid{0000-0002-9487-7545}}
\address{College of Education, University of the Ryukyus, Nishihara  903-0213, Japan}
\email{hide@edu.u-ryukyu.ac.jp}
\urladdr{http://www.cc.u-ryukyu.ac.jp/~hide/}

\begin{document}
\begin{abstract}
The Hamming graph $H(n,q)$ is defined on the vertex set $[q]^n$ and
two vertices are adjacent if and only if they differ in precisely one 
coordinate.
Alon \cite{Alon} proved that the burning number of $H(n,2)$ is 
$\lceil\frac n2\rceil+1$.
In this note we give a short proof of a fact that the burning number of 
$H(n,q)$ is $(1-\frac 1q)n+O(\sqrt{n\log n})$ for fixed $q\geq 2$ and $n\to\infty$.
\end{abstract}

\maketitle
\hypersetup{pdfborder={0 0 1}} 

Graph burning has been actively studied in recent years as a model for the 
spread of influence in a network, see e.g., \cite{Bonato}.
One of the seminal results due to Alon determines the burning number of the
$n$-dimensional cube. In this note we present a short proof yielding
an asymptotic bound of the burning number of the Hamming graph.

Let us define the burning number of a finite graph $G$ with the vertex set 
$V$. For vertices $x,y\in V$ let $d(x,y)$ denote the distance between 
$x$ and $y$.
For a non-negative integer $k$, 
let $\Gamma_k(x)$ denote the $k$-neighbors of $x$, that is,
the set of vertices $y\in V$ such that $d(x,y)\leq k$.
We say that $x_0,x_1,\ldots,x_b\in V$ is a burning sequence of length $b+1$
if $\bigcup_{k=0}^b\Gamma_{b-k}(x_k)=V$. 
The burning number of $G$, denoted by $\beta(G)$, is defined to be the minimum 
length of the burning sequences.

A path with $n$ vertices has burning number $\lceil\sqrt n\rceil$,
and it is conjectured that $\beta(G)\leq\lceil\sqrt n\rceil$ for every
$n$-vertex graph $G$. Although this is one of the major open problems 
concerning graph burning, it is also interesting to find graphs which
have small burning number and small maximum degree. We will show that
Hamming graphs satisfy these conditions.

Recall some basic facts about Hamming graphs.
For positive integers $n,q$, 
let $[q]:=\{0,1,\ldots,q-1\}$ and let $[q]^n$ denote the set of 
$n$-tuple of elements of $[q]$. For $x\in [q]^n$ we write $(x)_i$ for
the $i$th coordinate of $x$. The Hamming distance between 
$x,y\in[q]^n$ is defined to be $\#\{i:(x)_i\neq(y)_i\}$.
The Hamming graph $H(n,q)$ has the vertex set $V=[q]^n$, and
two vertices are adjacent if they have Hamming distance one. 
Thus $H(n,q)$ is $(q-1)n$-regular, and 
$|\Gamma_k(x)|=\sum_{i=0}^k(q-1)^i\binom ni$ is independent of the
choice of $x\in V$. Note that $H(n,2)$ is the $n$-dimensional cube. 

Alon considered a problem of transmitting on the $n$-dimensional cube.
His result is restated in terms of burning number as follows.
\begin{thm}[Alon \cite{Alon}]
$\beta(H(n,2))=\lceil\frac n2\rceil + 1 =\lfloor\frac n2+\frac32\rfloor$.
\end{thm}
\noindent
It is easy to see that $\beta(H(n,2))\leq b+1$, where $b:=\lceil\frac n2\rceil$.
Indeed, by choosing $x_0=(0,\ldots,0)$ and $x_1=(1,\ldots,1)$, we see that
$\Gamma_b(x_0)\cup\Gamma_{b-1}(x_1)=[2]^n$. (Thus you can choose
$x_2,\ldots,x_b$ arbitrarily.) On the other hand,
it is non-trivial to see that $b+1$ is the correct lower bound.
To this end, Alon used the Beck--Spencer Lemma \cite{Beck-Spencer} 
which is based on a result of Beck and Fiala \cite{Beck-Fiala} for 
combinatorial discrepancies.

Now we show that $\beta(H(n,q))=(1-\frac 1q)n+O(\sqrt{n\log n})$ for fixed
$q\geq 2$ and $n\to\infty$. We mention that in \cite{T2024} the error term
is improved to $O(1)$, but the proof is much more complicated.

\begin{thm}
We have
\begin{align*}
\beta(H(n,q))\leq \lfloor(1-\tfrac1q)n+\tfrac{q+1}2\rfloor. 
\end{align*}
\end{thm}
\begin{proof}
Let $n=qk+r$ where $r\in[q]$, and let 
$s:=\lceil r-\frac rq+\frac{q-1}2+\frac1{2q}\rceil$.
Let $b+1=(q-1)k+s$, and let $x_i=(i,\ldots,i)$ for $i\in[q]$. 
Note that $x\in\Gamma_{b}(x_0)$ if and only if 
$\#\{j:(x)_j\neq 0\}\leq b$, or equivalently, 
$\#\{j:(x)_j=0\}\geq n-b=k+r-s+1$.
Similarly, $x\in\Gamma_{b-i}(x_i)$ if and only if 
$\#\{j:(x)_j=i\}\geq k+r-s+1+i$ for $i\in[q]$. 
We claim that $W:=\bigcup_{i=0}^{q-1}\Gamma_{b-i}(x_i)=[q]^n$.
To the contrary, suppose that $y\in[q]^n$ is not covered by $W$.
Then $\#\{j:(y)_j=i\}\leq k+r-s+i$ for all $i\in[q]$. 
By summing both sides from $i=0$ to $q-1$, we get
$n=qk+r\leq(k+r-s)q+\binom q2$, or equivalently,
$s\leq r-\frac rq+\frac{q-1}2$, a contradiction.
Thus we have $\beta(H(n,q))\leq b+1=(q-1)\frac{n-r}q+s
=\lfloor(1-\tfrac1q)n+\tfrac{q+1}2\rfloor$, where we used
$\lceil N/(2q)\rceil =\lfloor (N+2q-1)/(2q)\rfloor$ for $N\in\Z$
in the last equality.
\end{proof}

\begin{thm}
Let $p=1-\frac1q$. Then we have
\[
 \beta(H(n,q))>pn-\sqrt{2pn\log n}.
\]
\end{thm}

\begin{proof}
Let $b=\lfloor pn-\sqrt{2pn\log n}\rfloor$.
We need to show that no matter how $x_0,\ldots,x_b\in V$ are chosen, we have
$|\bigcup_{i=0}^b\Gamma_{b-i}(x_i)|<|V|$.
Even more strongly, we claim that $(b+1)|\Gamma_b(x_0)|<|V|$.

Let $X_1,\ldots,X_n$ be independent random variables with
$\P[X_i=1]=p$, $\P[X_i=0]=1-p$.
Let $X=\sum_{i=1}^nX_i$. Then we have $\mu:=\E[X]=pn$ and
$\P[X=k]=\binom nkp^k(1-p)^{n-k}$.
It follows from the Chernoff bound \cite{Chernoff} (see also e.g., 
\S 21.4 in \cite{Frieze-Karonski}), by setting 
$\epsilon=\sqrt{(2\log n)/(pn)}$, that
\begin{align*}
\P\left[X\leq pn-\sqrt{2pn\log n}\right]=
\P\left[X\leq(1-\epsilon)\mu\right]\leq 
\exp\left(-\frac{\mu\epsilon^2}2\right)=
\frac1n. 
\end{align*}
Thus we have
\begin{align*}
 q^{-n}|\Gamma_b(x_0)|
&=q^{-n}\sum_{k=0}^b(q-1)^k\binom nk
=\sum_{k=0}^b\binom nkp^k(1-p)^{n-k}\\
&=\P[X\leq b]
\leq\P\left[X\leq pn-\sqrt{2pn\log n}\right]
\leq \frac1n, 
\end{align*}
and so
\[
(b+1)|\Gamma_b(x_0)|\leq(b+1)\frac{|V|}n<p|V|<|V|,
\]
as desired.
\end{proof}

\section*{Acknowledgments}
The author thanks Naoki Matsumoto for stimulating discussions.
The author was supported by JSPS KAKENHI Grant Number JP23K03201.

\section*{Data Availability}
Data sharing not applicable to this article as no data sets were generated 
or analysed during the current study.

\section*{Declarations}
Conflict of interest: 
The author has no relevant financial or non-financial interests to disclose.

\end{document}